\documentclass{article}

\usepackage{mathtools}
\usepackage{amssymb}
\usepackage{amsthm}
\usepackage{mytheorems}
\usepackage{mymacros}

\theoremstyle{plain}
\newcommand{\thmDT}{Dickson's Theorem}
\newtheorem*{TheoremDT}{\thmDT}

\begin{document}

\title{\bf On a theorem of Dickson}

\author{Paul Flavell}

\maketitle

\section{Introduction} \label{intro}

The purpose of this paper is to give a new proof of the following result.
Another new proof is given in \cite{F1}.

\begin{TheoremDT}[\cite{D}] \label{thmDT}
    Suppose that $\openF$ is a finite field with odd characteristic
    and generator $l$.
    Set \[
        G = \bblistgen{ \PMatrixC{1 & 0 \\ 1 & 1}, \PMatrixC{1 & l \\ 0 & 1} }.
    \]
    Then either
    \begin{itemize}
        \item[(a)]  $G = \spl{2}{\openF}$ or

        \item[(b)]  $\card{\openF} = 9, l^{2} = -1$ and
                    $G$ is isomorphic to the double cover of $\alt{5}$.
    \end{itemize}
\end{TheoremDT}

\noindent Gorenstein~\cite[Theorem~8.4, p.44]{G} gives an account of Dickson's proof
in modern terminology.
Dickson's Theorem is in fact part of a larger work of Dickson that determines the
subgroups of the groups $\psl{2}{q}$.
Alternative treatments of this work are given by Suzuki
and by Huppert and Blackburn.
Dickson's Theorem is deduced as a corollary.
See \cite[6.21, p.409]{S}, \cite[7.6, p.494]{HB} and \cite[8.27, p.213]{H}.
This is also the approach taken in \cite{F1}.

The exception in (b) follows without much difficulty from
the isomorphism $\psl{2}{9} \isom \alt{6}$.
This does beg the question as to how the exception manifests itself
in any proof of Dickson's Theorem.
In the existing proofs and in \cite{F1},
it emerges out of counting arguments.
The way it emerges in the present paper is entirely different.

\section{Basics} \label{basics}
Throughout this paper \[
    \mbox{$\openF$ is a finite field with odd characteristic}.
\]
The \emph{projective line} over $\openF$ is
the set of $1$-dimensional subspaces of $\openF^{2}$
and is denoted by $\pl{\openF}$.
For each $(x,y) \in \openF^{2} \setminus{ \listset{(0,0)} }$
let $[x:y]$ denote the $1$-dimensional subspace spanned by $(x,y)$.
Then $[x:y] = [x':y']$ if and only if $xy' = yx'$.
Consequently \[
    \pl{\openF} = \set{ [x:1] }{ x \in \openF } \cup \listset{ [1:0] }.
\]
The group $\gl{2}{\openF}$ acts by right multiplication
on the $\openF$-vectorspace $\openF^{2}$.
This gives an action of $\gl{2}{\openF}$ on $\pl{\openF}$.
This action has kernel $\zenter{\gl{2}{\openF}}$.
Thus $\pgl{2}{\openF}$ acts faithfully on $\pl{\openF}$.

On occasion it is better to work in $\pgl{2}{\openF}$.
We thus use the \emph{bar convention}
namely if $G \leq \gl{2}{\openF}$ then \[
    \mbox{ $\BAR{G}$ denotes the image of $G$ in $\pgl{2}{\openF}$.}
\]
Thus $\BAR{G} \not= 1$ if and only if $G$ acts nontrivially on $\pl{\openF}$.

We use the standard subscript notation $G_{\alpha}$ and $G_{\alpha\beta}$ to
denote $1$-point and $2$-point stabilizers respectively.

Let $\alpha \in \pl{\openF}$.
The stabilizer $\gl{2}{\openF}_{\alpha}$ acts on the subspace $\alpha$
and the quotient space $\openF^{2} / \alpha$.
Define \[
    \unip{\alpha}{\openF} = %
        \set{ g \in \gl{2}{\openF}_{\alpha} }{%
            \mbox{$g$ acts trivially on $\alpha$ %
                    and on $\openF^{2} / \alpha$
            }
        }.
\]
Then $\unip{\alpha}{\openF} \normal \gl{2}{\openF}_{\alpha}$.
In order to study $\unip{\alpha}{\openF}$ we may,
by the transitivity of $\gl{2}{\openF}$ on $\pl{\openF}$,
suppose that $\alpha = [1:0]$.
Let $g \in \gl{2}{\openF}_{\alpha}$ so that \[
    g = \PMatrixC{ a & 0 \\ c & d }
\]
for some $a,d \in \openF \setminus \listset{0}$ and $c \in \openF$.
Now $g$ acts as
multiplication by $a$ on $\alpha$ and
as multiplication by $d$ on $\openF^{2} / \alpha$.
Hence \[
    \unip{\alpha}{\openF} =%
        \bbset{ \PMatrixC{ 1 & 0 \\ c & 1} }{ c \in \openF}.
\]
Then $\unip{\alpha}{\openF} \isom \openF\additive$ and $\unip{\alpha}{\openF} \leq \spl{2}{\openF}$.
Also $[0:1]\PSmallMatrixC{ 1 & 0 \\ c & 1 } = [c:1]$
so the action of $\unip{\alpha}{\openF}$ on
$\pl{\openF} \setminus \listset{\alpha}$ is regular.
This implies that $\spl{2}{\openF}$ is $2$-transitive on $\pl{\openF}$.
Every element of $\BAR{\unip{\alpha}{\openF}}\nonid$ is
fixed point free on $\pl{\openF} \setminus \listset{ \alpha }$.
By contrast if $\BAR{g} \not\in \BAR{\unip{\alpha}{\openF}}$ then
$\BAR{g}$ has a unique fixed point on $\pl{\openF} \setminus \listset{\alpha}$
namely $[ (d - a)^{-1}c : 1 ]$.

The subgroups $\unip{\alpha}{\openF}$ are the maximal unipotent subgroups
of $\gl{2}{\openF}$ as well as being the Sylow $\Char\openF$-subgroups
of $\gl{2}{\openF}$ and of $\spl{2}{\openF}$.

\section{Preliminaries} \label{prel}

\begin{Lemma} \label{prel.1}
    \[
        \spl{2}{\openF} = \bbgen{
            \PMatrixC{ 1 & 0 \\ 1 & 1},
            \PMatrixC{ 1 & f \\ 0 & 1}
        }{ f \in \openF}.
    \]
\end{Lemma}
\begin{proof}
    By elementary row and column operations.
\end{proof}

\begin{Lemma} \label{prel.2}
    Let $\openE_{0} \subseteq \openE \subseteq \openF$
    where $\openE_{0}$ and $\openE$ are subgroups of $\openF\additive$.
    Suppose that $e^{2} \in \openE$ whenever $e \in \openE \setminus \openE_{0}$.
    Then $\openE_{0} = \openE$ or $\openE$ is a subfield of $\openF$.
\end{Lemma}
\begin{proof}
    We assume that $\openE_{0} \not= \openE$.
    Since $\openF$ is finite it suffices to show that $\openE$
    is closed under multiplication.
    Let $e \in \openE \setminus \openE_{0}$ and $f \in \openE_{0}$.
    Then $e \pm f \not\in \openE_{0}$ so \[
        2f^{2} = (e + f)^{2} + (e - f)^{2} - 2e^{2} \in \openE.
    \]
    Since $\openF$ has odd characteristic and $\openE \leq \openF\additive$
    this forces $f^{2} \in \openE$.
    We deduce that $e^{2} \in \openE$ for all $e \in \openE$.
    Finally let $e,f \in \openE$ be arbitrary.
    Then \[
        2ef = (e + f)^{2} - e^{2} - f^{2} \in \openE
    \]
    and again $ef \in \openE$,
    completing the proof.
\end{proof}


\noindent The following lemma follows readily from the conjugacy of Frobenius complements.
We give one of several elementary proofs.

\begin{Lemma} \label{prel.3}
    Let $N \leq \gl{2}{\openF}$ with $\zenter{\gl{2}{\openF}} \leq N$.
    Let $\alpha \in \pl{\openF}$,
    set \[
        \Sigma = \bset{  \beta \in \pl{\openF} \setminus \listset{ \alpha } }{
                    \BAR{N}_{\alpha\beta} \not= 1
                 }
    \]
    and suppose that $\Sigma \not= \emptyset$.
    Then $N \cap \unip{\alpha}{\openF}$ acts transitively on $\Sigma$.
\end{Lemma}
\begin{proof}
    Let $\Omega = \pl{\openF} \setminus \listset{ \alpha }$ and
    consider the action of $\BAR{N}_{\alpha}$ on $\Omega$.
    Then $\Sigma$ is the union of the nonregular orbits.
    Let $n$ be the number of such orbits and
    set $\BAR{U} = \BAR{N} \cap \BAR{\unip{\alpha}{\openF}} \leq \BAR{N}_{\alpha}$.
    Let $\BAR{g} \in \pgl{2}{\openF}_{\alpha}$.
    Recall that if $\BAR{g} \in \BAR{ \unip{\alpha}{\openF} }\nonid$
    then $\BAR{g}$ is fixed point free on $\Omega$
    whilst if $\BAR{g} \not \in \BAR{ \unip{\alpha}{\openF} }$
    then $\BAR{g}$ has a unique fixed point on $\Omega$.
    The Orbit-Counting Lemma gives \[
        n\card{\BAR{N}_{\alpha}} = \sum_{\BAR{g} \in \BAR{N}_{\alpha}} \card{ \fix{\Sigma}{\BAR{g}} } = %
        \card{\Sigma} + \card{\BAR{U}\nonid} \cdot 0 + ( \card{\BAR{N}_{\alpha}} - \card{\BAR{U}} ) \cdot 1
    \] whence $(n - 1) \card{\BAR{N}_{\alpha}} = \card{\Sigma} - \card{\BAR{U}}$.

    Each nonregular orbit has length at most $\half\card{\BAR{N}_{\alpha}}$
    so $\card{\Sigma} \leq \half n\card{\BAR{N}_{\alpha}}$.
    Consequently $n - 1 < \half n$.
    Then $n = 1$ and $\card{\Sigma} = \card{\BAR{U}}$.
    Since $\BAR{\unip{\alpha}{\openF}}$ is regular on $\Omega$ it follows that
    $\BAR{U}$ is transitive on $\Sigma$.
    Finally, $\zenter{\gl{2}{\openF}} \leq N$ so $N \cap \unip{\alpha}{\openF}$
    maps onto $\BAR{U}$,
    completing the proof.
\end{proof}

\begin{Lemma} \label{prel.4}
    Let $\alpha, \beta \in \pl{\openF}$ be distinct.
    Then there exists $n \in \gl{2}{\openF}_{\alpha\beta}$
    with $\BAR{n} \not= 1$ such that \[
        a^{n} = a^{-1} \qtext{and} b^{n} = b^{-1}
    \]
    for all $a \in \unip{\alpha}{\openF}$ and $b \in \unip{\beta}{\openF}$.
\end{Lemma}
\begin{proof}
    Since $\gl{2}{\openF}$ is $2$-transitive on $\pl{\openF}$
    we may suppose that $\alpha = [1:0]$ and $\beta = [0:1]$.
    Let $n = \PSmallMatrixC{-1 & 0 \\ 0 & 1}$.
\end{proof}

\begin{Lemma} \label{prel.5}
    Suppose that $\card{\openF} = 9$ and
    that $l \in \openF$ is a square root of $-1$.
    Let \[
        G = \bblistgen{ \PMatrixC{1 & 0 \\ 1 & 1}, \PMatrixC{1 & l \\ 0 & 1} },
        \quad \alpha = [1:0] \qtext{and} \beta = [0:1].
    \]
    Then $G/\zenter{G} \isom \alt{5},\, \zenter{G} = \listset{\pm 1}$,\,
    $G$ is perfect and\, $\BAR{G}_{\alpha\beta} \not= 1$.
\end{Lemma}
\begin{proof}
    Let $g = \PSmallMatrixC{1 & 1 \\ 0 & 1}, h = \PSmallMatrixC{1 & l \\ 0 & 1}$,
    $t = \PSmallMatrixC{0 & -1 \\ 1 & 0}$ and
    let $\BAR{g}, \BAR{h}$ and $\BAR{t}$ denote the images of these matrices in $\psl{2}{\openF}$.
    By \cite[p.52]{W} there is an isomorphism $\theta : \psl{2}{\openF} \longrightarrow \alt{6}$
    with $\BAR{g}\theta = (123), \; \BAR{h}\theta = (456)$ and
    $\BAR{t}\theta = (23)(14)$.
    Note that $g^{t} = \PSmallMatrixC{1 & 0 \\ -1 & 1}$.
    Let $\Omega = \listset{2,3,4,5,6}$.
    Then \[
        \BAR{G}\theta = \blistgen{ \BAR{g}^{\BAR{t}}, \BAR{h} }\theta = %
        \blistgen{ (432), (456) } = \alt{\Omega} \isom \alt{5}.
    \]
    The kernel of the map $\spl{2}{\openF} \longrightarrow \pgl{2}{\openF}$ sending $k$ to $\BAR{k}$
    is $\listset{ \pm 1 }$ and this is the unique subgroup of order $2$ in $\spl{2}{\openF}$.
    Since $\alt{5}$ is simple,
    the first three assertions follow.
    Let $n = (32)(56) \in \alt{\Omega}$.
    Now $n$ inverts $(432)$ and $(456)$.
    Since $\alpha$ and $\beta$ are the unique fixed points of $\BAR{g}^{\BAR{t}}$
    and $\BAR{h}$ respectively,
    it follows that $1 \not= n\theta^{-1} \in \BAR{G}_{\alpha\beta}$.
\end{proof}

\section{The proof of \thmDT} \label{proof}
Assume the hypotheses of \thmDT,
set $\alpha = [1:0], \beta = [0:1]$ and
$a_{0} = \PSmallMatrixC{1 & 0 \\ 1 & 1} \in G \cap \unip{\alpha}{\openF}$.
Define $B \leq \unip{\beta}{\openF}$ and $\openE \subseteq \openF$ by \[
    B = G \cap \unip{\beta}{\openF} = \bbset{ \PMatrixC{ 1 & e \\ 0 & 1} }{ e \in \openE }.
\]
Then $l \in \openE$ and $G = \listgen{a_{0},B}$.

Suppose that $\openE$ is a field.
Now $l$ generates $\openF$ so $\openE = \openF$ and
Lemma~\ref{prel.1} implies that $G = \spl{2}{\openF}$.
Then Conclusion (a) holds.
Hence we assume that
\begin{equation} \tag{$1$}
    \mbox{ $\openE$ is not a field}.
\end{equation}
On the other hand,
the definition of matrix multiplication implies that
\begin{equation} \tag{$2$}
    \mbox{ $\openE$ is a subgroup of $\openF\additive$}.
\end{equation}
It is necessary to consider a larger group.
Let \[
    N = \nn{\gl{2}{\openF}}{G}.
\]
The element $n$ of Lemma~\ref{prel.4} inverts the generators for $G$ so
\begin{equation} \tag{$3$}
    \BAR{N}_{\alpha\beta} \not= 1.
\end{equation}
Lemma~\ref{prel.2}, with $\openE_{0} = \listset{0}$,
enables us to choose $e \in \openE \setminus \listset{0}$ with
\begin{equation} \tag{$4$}
    e^{2} \not\in \openE.
\end{equation}
Let $b = \PSmallMatrixC{1 & e \\ 0 & 1} \in B\nonid$ and
$a = \PSmallMatrixC{1 & 0 \\ e^{-1} & 1} \in \unip{\alpha}{\openF}$.
Then $\alpha b = [1:e] = [e^{-1}:1] = \beta a$.
Since $\unip{\alpha}{\openF}$ is regular on $\pl{\openF} \setminus \listset{\alpha}$
it follows that
\begin{equation} \tag{$5$}
    \mbox{ $a$ is the unique element of $\unip{\alpha}{\openF}$ with $\alpha b = \beta a$}.
\end{equation}
Then \[
    \alpha ba^{-1} = \beta \qtext{and} %
    ba^{-1} = \PMatrixC{ 1 & e \\ 0 & 1 } \PMatrixC{ 1 & 0 \\ -e^{-1} & 1 } = %
              \PMatrixC{ 0 & e \\ -e^{-1} & 1 }.
\]
Moreover
\begin{align*}
    a_{0}^{ba^{-1}} &= \PMatrixC{ 1 & -e \\ e^{-1} & 0 } \PMatrixC{ 1 & 0 \\ 1 & 1 } \PMatrixC{ 0 & e \\ e^{-1} & 1} %
    = \PMatrixC{ 1-e & -e \\ e^{-1} & 0} \PMatrixC{ 0 & e \\ -e^{-1} & 1 } \\
    &= \PMatrixC{ 1 & -e^{2} \\ 0 & 1} \in \unip{\beta}{\openF}.
\end{align*}
Now $-e^{2} \not\in \openE$ by $(4)$ and $(2)$ so $a_{0}^{ba^{-1}} \not\in G$.
As $a_{0}^{b} \in G$ we have $a \not\in N$.
If $\BAR{N}_{\alpha \, \alpha b} \not= 1$ then
$(3)$, Lemma~\ref{prel.3} and $(5)$ force $a \in N$.
We deduce that
\begin{equation} \tag{$6$}
    \BAR{N}_{\alpha \, \alpha b} = 1.
\end{equation}

Let $H = \listgen{ a_{0}, a_{0}^{b} } \leq G$.
Note that $a_{0}^{b} \in \unip{\alpha b}{\openF}$.
If $H = G$ then Lemma~\ref{prel.4},
with $\alpha b$ in the role of $\beta$,
implies that $\BAR{N}_{\alpha \, \alpha b} \not= 1$,
contrary to $(6)$.
Thus $H < G$.

Let \[
    L = H^{a^{-1}} = \blistgen{ a_{0}, a_{0}^{ba^{-1}} } = %
    \bblistgen{ \PMatrixC{ 1 & 0 \\ 1 & 1}, \PMatrixC{ 1 & -e^{2} \\ 0 & 1 } }.
\]
Now $\BAR{H}_{\alpha\, \alpha b} \leq \BAR{N}_{\alpha\, \alpha b} = 1$,
$\alpha a^{-1} = \alpha$ and $\alpha ba^{-1} = \beta$ so
\begin{equation} \tag{$7$}
    \BAR{L}_{\alpha \, \beta} = 1.
\end{equation}

Since $\card{L} < \card{G}$ we may use induction to identify $L$.
Lemma~\ref{prel.5} and $(7)$ rule out Conclusion~(b).
Let $\openD$ be the subfield of $\openF$ generated by $-e^{2}$.
Then \[
    L = \spl{2}{\openD}.
\]
We have \[
    L_{\alpha\beta} = \bbset{ \PMatrixC{ d & 0 \\ 0 & d^{-1} } }{ d \in \openD \setminus \listset{0} } \isom \openD\mult
\]
whence $\BAR{L}_{\alpha\beta} \isom \openD\mult / \listset{\pm 1}$.
Recall that $\Char\openF$ is odd.
Then $(7)$ forces $\card{\openD} = 3$ so
$e^{2} \in \openD \setminus \listset{0} = \listset{ \pm 1}$.
If $e^{2} = 1$ then $e \in \listset{ \pm 1 }$ and
$(2)$ forces $1 \in \openE$,
contrary to $e^{2} \not\in \openE$.
We deduce that \[
    e^{2} = -1.
\]

Let $\openE_{0} = \listset{0, e, -e}$.
Now $\card{\openD} = 3$ so $\Char\openF = 3$ and
$\openE_{0}$ is a subgroup of $\openF\additive$.
Note that $e$ and $-e$ are the only square roots of $-1$ in $\openF$.
We conclude that \[
    \mbox{ $f^{2} \in \openE$ for all $f \in \openE \setminus \openE_{0}$ }.
\]
Since $\openE$ is not a field,
Lemma~\ref{prel.2} forces $\openE_{0} = \openE$.
Now $l \in \openE$ so $l \in \listset{ \pm e }$ and $l^{2} = -1$.
Since $l$ generates $\openF$ and $\Char\openF = 3$ we have
$\card{\openF} = 9$.
Then Conclusion~(b) follows from Lemma~\ref{prel.5}
and the proof of Dickson's Theorem is complete. 

\bibliographystyle{amsalpha}

%
%


\end{document}